\documentclass{article}

\usepackage{enumerate}
\usepackage{amsthm}
\usepackage{amsbsy}
\usepackage{amscd}
\usepackage[dvips]{graphicx}
\usepackage{amsmath,amscd,amssymb,epsfig,subfig}%,pstricks}
\usepackage[all]{xy}
\usepackage{verbatim}
\usepackage{graphics,graphicx}
\usepackage{tikz}
\theoremstyle{plain}
\newtheorem{theorem}{Theorem}[section]
\newtheorem{lemma}{Lemma}[section]

\newtheorem{proposition}{Proposition}[section]

\newtheorem{definition}{Definition}[section]
\theoremstyle{remark}
\newtheorem{remark}{Remark}[section]

\allowdisplaybreaks

\begin{document}

\date{\today}

\title{Central Extension of Mapping Class Group via Chekhov-Fock Quantization}
  \author{Binbin XU \\
\small \em Institut Fourier\\
\small \em 100 Rue Des Maths \\
\small \em 38402 Saint-Martin-d'H\`eres cedex, France\\
\small \em e-mail: {\tt binbin.xu@ujf-grenoble.fr}\\
}

\maketitle

\begin{abstract}
The central extension of mapping class groups of punctured surfaces of finite type that arises in Chekhov-Fock quantization is 12 times of the Meyer class plus the Euler classes of the punctures, which agree with the one arising in the Kashaev quantization.
\end{abstract}

\section{Introduction}
The quantum theory of Teichm\"uller spaces of punctured surfaces was developed in \cite{Fock} and \cite{Kashaev} independently, and then generalized to the higher rank Lie groups and cluster algebras in \cite{FGTeich} and \cite{FGcluster}. The main ingredient of both the constructions is Faddeev's quantum dilogarithm introduced in \cite{Faddeev}. This theory leads to one parameter families of projective unitary representations of Ptolemy groupoids associated to ideal triangulations of punctured surfaces. 

Let $V$ be a vector space and $G$ be a group. A \textit{projective representation} of $G$ on $V$ is a homomorphism from $G$ to $PGL(V)$. It is well-known that a projective representation of a group gives rise to a linear representation of some central extension of the given group. More precisely, let $h$ be a projective representation of $G$ on $V$. Let $\widetilde G$ be a central extension of $G$ by $\mathbb{C}^*$ which is the pullback of $GL(V)\rightarrow PGL(V)$ by $h$. Then one can associate a representation $\widetilde{h}$ of $\widetilde{G}$ on $V$, such that the following commutative diagram holds :

\centerline{
\xymatrix{
1\ar[r] & \mathbb{C}^*\ar[r] & GL(V) \ar[r] & PGL(V) \ar[r] & 1\\
1\ar[r] & \mathbb{C}^* \ar[r]\ar[u]  & \widetilde{G} \ar[r]\ar[u]_{\widetilde{h}} & G \ar[r]\ar[u]_h & 1
}}

A \textit{reduction} $\widetilde G_1$ of $\widetilde G$ is a central extension of $G$ by a subgroup $A_1$ of $\mathbb{C}^*$, such that $\widetilde G_1$ is a subgroup of $\widetilde G$ and the associated representation $\widetilde h_1:\widetilde G_1\rightarrow GL(V)$ is the restriction of $\widetilde h$. We say that $\widetilde G_1$ is the \textit{minimal reduction} of $\widetilde G$ if $\widetilde G_1$ is minimal among all the reductions of $\widetilde G$ up to isomorphism.

Suppose that $G$ has the presentation $F/R$ where $F$ is a free group and $R$ is the normal subgroup of $F$ generated by a set of relations. A projective representation of $G$ on $V$ is induced by a linear representation $\bar h$ of $F$ on $V$ such that $R$ is sent to the center of $GL(V)$. The homomorphism $\bar{h}$ will be called an \textit{almost linear representation} of $G$ on $V$, in order to distinguish the projective representation.

Let $\Sigma^{s}_{g}$ be an oriented compact surface of genus $g$ with $s$ punctures where $s>0$. Let $\Gamma(\Sigma^{s}_{g})$ be its mapping class group. An element of $\Gamma(\Sigma^{s}_{g})$ is a mapping class of homeomorphisms from $\Sigma^{s}_{g}$ to itself fixing the punctures setwisely. Let $\mathcal{H}$ denote the Hilbert space $L^2(\mathbb{R}^{2n})$ where $n$ is the number of arcs in an ideal triangulation of $\Sigma^{s}_{g}$. The Chekhov-Fock quantization constructs a family of projective representations $\rho_z:\Gamma(\Sigma^{s}_{g})\rightarrow PGL(\mathcal{H})$ depending on one parameter $z$ with $|z|=1$. Denote by $\widetilde{\mathbf{\Gamma}(\Sigma^{s}_{g})}$ the central extension $\rho_z^*\eta$ where $\eta$ is the central extension $GL(\mathcal{H})\rightarrow PGL(\mathcal{H})$.  

Central extensions of a group $G$ by an abelian group $A$ are classified, up to isomorphism, by elements of the 2-cohomology group $H^2(G,A)$. For the mapping class group $\Gamma(\Sigma^{s}_{g})$ with $g\ge4$ and $s\ge2$, by the work of Harer in \cite{Harer1} and Korkmaz and Stipsicz in \cite{Kork}, we have:
\begin{equation*}
H^2(\Gamma(\Sigma^{s}_{g}),\mathbb{Z})=\mathbb{Z}^{1+s},
\end{equation*} 
where the generators are given by one fourth of the Meyer signature class $\chi$ and $s$ Euler classes $e_i$ that are associated to $s$ punctures $p_i$ respectively. For $g=3$, the group $H^2(\Gamma(\Sigma^{s}_{g}),\mathbb{Z})$ still contains $\mathbb{Z}^{s+1}$, but it is not known if there are any other classes. For $g=2$, it has been proved that $H^2(\Gamma(\Sigma^{s}_{g}),\mathbb{Z})$ contains the subgroup $\mathbb{Z}/10\mathbb{Z}\oplus\mathbb{Z}^{S}$ whose torsion part is generated by $\chi$ and whose free part is generated by the Euler classes.

Our main result is the following theorem:
\begin{theorem}\label{main}
 Let $g\ge2$ and $s\ge4$. The minimal reduction $\widetilde{\Gamma(\Sigma^{s}_{g})}$ of $\widetilde{\mathbf{\Gamma}(\Sigma^{s}_{g})}$ is a central extension of $\Gamma(\Sigma^{s}_{g})$ by the cyclic subgroup $A\subset\mathbb{C}^*$ generated by $z^{-12}$, whose cohomology class is
\begin{equation*}
 c_{\widetilde{\Gamma(\Sigma^{s}_{g})}}=12\chi+\sum_{i=1}^s e_i\in H^2(\Gamma(\Sigma^{s}_{g}),A).
\end{equation*}
\end{theorem}

The organization of this paper is the following: in section 2, we will briefly recall some basic definitions and the Chekhov-Fock quantization, and in section 3 we will give the proof of our main theorem.

\subsection*{Acknowledgements} The author is greatly indebted to Louis Funar for suggesting this problem and for helpful comments. We would like to thank Vlad Sergiescu for useful discussions. The author acknowledges the support of the French Agence Nationale de la Recherche (ANR) ANR-11-BS01-0020 .

\section{Preliminaries}
\subsection{Ptolemy groupoids}
Let $\Sigma=\Sigma^s_g$ and $\Gamma=\Gamma(\Sigma^s_g)$.
\begin{definition}
 An \textbf{arc} is the homotopy class of a simple curve on $\Sigma$ connecting punctures which is non homotopic to a point or a puncture of $\Sigma$.
\end{definition}

\begin{definition}
 An \textbf{ideal triangulation} of $\Sigma$ is a maximal collection of distinct arcs which have pairwise disjoint representatives. A \textbf{labeled ideal triangulation} is obtained from an ideal triangulation by adding labels to its arcs. 
\end{definition}
We denote by $|\textbf{T}(\Sigma)|$ the set of ideal triangulations of $\Sigma$ and by $\textbf{T}(\Sigma)$ the set of labeled ideal triangulations of $\Sigma$.

Let $T$ be a labeled ideal triangulation. Let $\alpha$ be an arc of $T$ which is the common boundary of two distinct ideal triangles whose union is an embedded quadrilateral in $\Sigma$.

\begin{definition}
A \textbf{flip} on $\alpha$ in $T$ is to substitute $\alpha$ by the other diagonal $\alpha'$ of this ideal quadrilateral to get a new ideal triangulation $T'$ of $\Sigma$.
\end{definition}

We will use $F_\alpha(T)$ to denote the flip on the arc $\alpha$ in $T$.

A \textbf{groupoid} is a category such that all morphisms are invertible and for each pair of objects there exists at least one morphism between them. The set of automorphisms of an object forms a group. Reciprocally if a group $G$ acts freely on a set $X$, we can define an associated groupoid whose objects are the $G$-orbits in $X$ and whose morphisms are the orbits of the diagonal $G$-action on $X\times X$.

It is easy to see that the action of $\Gamma$ on $\textbf{T}(\Sigma)$ is free. By the above argument, we obtain a groupoid, called the \textbf{Ptolemy groupoid}, defined as follows:
\begin{enumerate}
 \item[(1)] the objects are the $\Gamma$-orbits in $\mathbf{T}(\Sigma)$;
 \item[(2)] the morphisms are the $\Gamma$-orbits in $\mathbf{T}(\Sigma)\times\mathbf{T}(\Sigma)$.
\end{enumerate}
\begin{remark}
When $s>1$, the $\Gamma$-action on $|\textbf{T}(\Sigma)|$ is not free.
\end{remark}

The Ptolemy groupoid can also be defined by using actions of flips and the symmetry group on $\mathbf{T}(\Sigma)$ where the symmetry group acts as permutations of labels. Then the Ptolemy groupoid has the following presentation due to Harer in \cite{Harer} and Penner in \cite{Penner1} and \cite{Penner2}:
\begin{theorem}\label{presentation}
If $\Sigma$ is different from the three-punctured sphere and the one-punctured torus, then any pair of labeled ideal triangulations can be connected by a chain of flips.

The Ptolemy groupoid is generated by the action of flips and the symmetry group. The relations between them are the following:
\begin{itemize}
\item[(1)] For any arc $\alpha$ in $T$, we have that $F_\alpha^2=1$;
\item[(2)] If $\alpha$ and $\beta$ are two arcs in $T$ having no common end point, we have that $F_\alpha F_\beta=F_\beta F_\alpha$;
\item[(3)] For any two arcs $\alpha$ and $\beta$ contained in an ideal pentagon, the pentagon relation holds: $$F_\alpha F_\beta F_\alpha F_\beta F_\alpha=\sigma(\alpha,\beta),$$ where $\sigma$ is the permutation of the labels of $\alpha$ and $\beta$;
\item[(4)] Let $\sigma\in S_n$ and let $\alpha$ be a labeled arc, then we have that $ F_\alpha \sigma=\sigma F_{\sigma(\alpha)}$.
\end{itemize}
\end{theorem}

\subsection{Shearing coordinates}
In this section, we recall the shearing coordinates system of the Teichm\"uller space $\mathcal{T}(\Sigma)$ of $\Sigma$ and the Poisson structure on it. 

We choose and fix an orientation on $\Sigma$. Given a hyperbolic structure on $\Sigma$, each arc has a unique geodesic representative. In the following, the word ``arc`` will denote its geodesic representative. Let $\alpha$ be an arc in an ideal triangulation $T$ of $\Sigma$. Fix an orientation on $\alpha$. Let $\widetilde\alpha$ be one of its lifts in the hyperbolic plane. Then $\widetilde\alpha$ will be an oriented diagonal in an ideal quadrilateral $Q$ with one ideal triangle $\Delta_l$ on its left side and another ideal triangle $\Delta_r$ on its right side. In each triangle, the vertex not lying on $\widetilde\alpha$ can be orthogonally projected on $\widetilde\alpha$. We denote by $v_l$ and $v_r$ the images of the vertices of $\Delta_l$ and $\Delta_r$ respectively. 

\begin{definition}
We call the directed hyperbolic distance from $v_l$ to $v_r$ with respect to the orientation of $\widetilde\alpha$ the \textbf{shearing coordinate} $t(\alpha)$ on $\alpha$. 
\end{definition}

\begin{remark}
Notice that lifts of $\alpha$ are different from each other by orientation preserving isometries of the hyperbolic plane, thus $t(\alpha)$ does not depend on the choice of $\widetilde\alpha$. At the same time, reversing the orientation of $\alpha$ also exchanges the left side and the right side of $\widetilde\alpha$. This implies that $t(\alpha)$ is also independent of the choice of the orientation of $\alpha$. So the shearing coordinate $t(\alpha)$ is well defined.
\end{remark}

The explicit formula of $t(\alpha)$ can be given by using cross-ratios of vertices of the associated ideal quadrilateral. Denote by $\{p_1,p_2,p_3,p_4\}$ the four vertices of the associated ideal quadrilateral $Q$ with a counter-clockwise order. Assuming that $p_1$ and $p_3$ are the vertices of $\widetilde\alpha$ with the orientation such that $p_2$ is on the left and $p_4$ is on the right. Then the shearing coordinate on $\alpha$ is defined by:
\begin{displaymath}
t(\alpha)=\log\left(-[p_1,p_3;p_2,p_4]\right)=\log\left(-\frac{(p_1-p_2)(p_3-p_4)}{(p_1-p_4)(p_3-p_2)}\right).
\end{displaymath}

The shearing coordinates system depends on the choice of the ideal triangulation. By doing a flip we get another shearing coordinates system. Let $T'$ be the ideal triangulation coming from $T$ by flipping $\alpha$ to $\alpha'$. By comparing the formulas of cross-ratios before and after the flip, we obtain the change formula:
\begin{displaymath}
 t'(\beta)=\left\{\begin{array}{ll}
-t(\alpha) & \textrm{if $\beta=\alpha'$}\\
t(\beta)+\epsilon(\alpha,\beta)\phi(\rm{sign}(\epsilon(\alpha,\beta))t(\alpha)) & \textrm{if $\beta$ and $\alpha'$ are adjacent but $\beta\neq\alpha'$}\\
t(\beta) & \textrm{otherwise}
\end{array}\right.,
\end{displaymath}
where $\phi(z)=\log(1+\exp(z))$ and the function $\epsilon_T$ is defined in the following. Let $\Delta$ be an ideal triangle on $\Sigma$ which is a connected component of $\Sigma\setminus T$. Let $E(\Delta)$ be the set of its edges. The orientation of $\Sigma$ induces an orientation of $\Delta$ which induces a cyclic order of the edges of $\Delta$, thus we can define an anti-symmetric map 
\begin{equation*}
\epsilon_{T,\Delta}:E(\Delta)\times E(\Delta)\rightarrow \{0,\pm1\}
\end{equation*}
in the following way: 
\begin{displaymath}
 \epsilon_{T,\Delta}(\alpha,\beta)=\left\{\begin{array}{ll}
-1 & \textrm{if $\beta$ comes after $\alpha'$ counter-clockwisely}\\
1 & \textrm{if $\beta$ comes after $\alpha'$ clockwisely}\\
0 & \textrm{otherwise}
\end{array}\right.,
\end{displaymath}
where $(\alpha,\beta)$ is in $E(\Delta)\times E(\Delta)$. By taking the sum of $\epsilon_{T,\Delta}$ over all ideal triangles $\Delta$, we obtain the following anti-symmetric map $\epsilon_T$:
\begin{equation*}
 \epsilon_T: E(T)\times E(T)\rightarrow \{0,\pm1,\pm2\},
\end{equation*}
where $E(T)$ is the set of arcs in $T$.

This anti-symmetric map $\epsilon_T$ also induces the Poisson structure on the Teichm\"uller space by the following bi-vector field:
\begin{displaymath}
 P(T)=\sum_{\alpha,\beta}\epsilon_T(\alpha,\beta)\frac{\partial}{\partial t(\alpha)}\wedge\frac{\partial}{\partial t(\beta)}.
\end{displaymath}

\subsection{Quantum Teichm\"uller space}

The quantization of a Poisson manifold equivariant with respect to a discrete group $G$-action is a family of $\ast$-algebras $A^\hbar$ depending smoothly on a positive real parameter $\hbar$ satisfying the following properties:
\begin{enumerate}[(1)]
\item All $A^\hbar$ are isomorphic to each other as linear spaces ;
\item The group $G$ acts as the outer automorphisms on each $\ast$-algebra ; 
\item For $\hbar=0$, the algebra $A^0$ is isomorphic as a $G$-module to the $\ast$-algebra of the complex-valued function on the Poisson manifold ; 
\item The Poisson bracket $\{,\}$ on $A^0$ is the limit of $\{,\}_{\hbar}/(2\pi i\hbar)$ as $\hbar$ going to zero. It coincides with the one on the original Poisson manifold.
\end{enumerate}

For each $\hbar>0$, we associate one $\ast$-algebra $A^\hbar(T)$ to each ideal triangulation $T$ on $\Sigma$, generated by $\{Z_\hbar(\alpha):\alpha\in T\}$ with the $\ast$- structure:
\begin{displaymath}
 (Z_\hbar(\alpha))^\ast=Z_\hbar(\alpha).
\end{displaymath}

The Poisson bracket on $A^\hbar(T)$ is obtained by deforming the Poisson bracket for the shearing coordinates associated to $T$ by the following formula:
\begin{displaymath}
 \{Z_\hbar(\alpha),Z_\hbar(\beta)\}_\hbar=2\pi i\hbar\{t(\alpha),t(\beta)\}.
\end{displaymath}

The flip acts on the $\ast$-algebras $A^\hbar(T)$ by the formula
\begin{displaymath}
 Z_\hbar'(\beta)=\left\{\begin{array}{ll}
-Z_\hbar(\alpha) & \textrm{if $\beta=\alpha'$}\\
Z_\hbar(\beta)+\epsilon(\alpha,\beta)\phi^\hbar(\rm{sign}(\epsilon(\alpha,\beta))Z_\hbar(\alpha)) & \textrm{if $\beta$ and $\alpha'$ are adjacent but $\beta\neq\alpha'$}\\
Z_\hbar(\beta) & \textrm{otherwise}
\end{array}\right.,
\end{displaymath}
where 
\begin{displaymath}
 \phi^\hbar(z)=-\frac{\pi\hbar}{2}\int_\Omega \frac{\exp(-iuz)}{\sinh(\pi u)\sinh(\pi\hbar u)}\rm{d}u,
\end{displaymath}
where $\Omega$ is the path going along the real axis from $-\infty$ to $+\infty$ and passing the origin from above. The symmetric group acts as permutations of labels.

For each $\hbar$, the construction above gives us a projective functor $\mathcal{Q}_\hbar$ from the Ptolemy groupoid to the category of $\ast$-algebra.
\begin{definition}
The family of projective functors $\mathcal{Q}_\hbar$ are called the \textbf{quantization} of Teichm\"uller space.
\end{definition}
\subsection{Almost linear representation of Ptolemy groupoid}
A Heisenberg $\ast$-algebra $H_n$ is generated by $2n+1$ generators $P_1,\dots,P_n,Q_1,\dots,Q_n$ and $C$ satisfying the following relations:
\begin{itemize}
 \item[(1)] The generator $C$ is the central element ;
 \item[(2)] For any two index $j$ and $k$, we have the relations: $$\{P_j,P_k\}=\{Q_j,Q_k\}=0;$$
 \item[(3)] For any two index $j$ and $k$, we have the relations: $$\{P_j,Q_k\}=C\delta_{jk}.$$
\end{itemize}

It has a irreducible integrable representation in the Hilbert space $\mathcal{H}$ described as following: consider the canonical complex structure on $\mathbb{R}^{2n}$ and denote by $z_1,\dots,z_n$ the complex coordinates and by $x_1,\dots,x_n,y_1,\dots,y_n$ the real coordinates such that $z_j=x_j+iy_j$, then we represent the generators as following operators:
\begin{eqnarray*}
 &&\rho(P_j)(f)(z_1,\dots,z_n)=z_jf(z_1,\dots,z_n),\\
 &&\rho(Q_j)(f)(z_1,\dots,z_n)=-2\pi i\hbar\frac{\partial f}{\partial z_j}(z_1,\dots,z_n),\\
 &&\rho(C)(f)(z_1,\dots,z_n)=2\pi i\hbar f(z_1,\dots,z_n).
\end{eqnarray*}

The representation of the $\ast$-algebra $A^\hbar(T)$ in $\mathcal{H}$ is generated by the linear combinations of $\rho(P_j)$ and $\rho(Q_k)$ above.

The Stone von Neumann theorem holds true for $A^\hbar(T)$. In particular, let $T$ and $T'$ be two labeled ideal triangulations different from each other by a flip $F_\alpha$. Consider the representations $\rho(A^\hbar(T))$ and $\rho(A^\hbar(T'))$ of $A^\hbar(T)$ and $A^\hbar(T')$ respectively. The uniqueness of representations yields the existence of an intertwinner $K(F_\alpha)$ between these two representations. It acts in the following way:
\begin{displaymath}
 \rho(A^\hbar(T'))=K(F_\alpha)^{-1}\rho(A^\hbar(T))K(F_\alpha).
\end{displaymath}

This intertwinner functor $K$ induces an almost linear representation of the Ptolemy groupoid which induces an almost linear representation of $\Gamma$ into a unitary group. The following result are proved in \cite{FGcluster}:
\begin{proposition}
 The almost linear representation $K$ has the following property:
 \begin{itemize}
  \item[(1)] For any disjoint arcs $\alpha$ and $\beta$ in the labeled ideal triangulation $T$, the operators $K(F_\alpha)$ and $K(F_\beta)$ commute with each other ;
  \item[(2)] $K(F_\alpha)^2=1$ ;
  \item[(3)] The pentagon relation: $$K(F_\alpha)K(F_\beta)K(F_\alpha)K(F_\beta)K(F_\alpha)=e^{2\pi i\hbar} \sigma,$$where $\sigma$ is the permutation of the labels $\alpha$ and $\beta$.
 \end{itemize}
\end{proposition}

\section{Proof of the theorem}

In \cite{Funar}, the authors gave the following presentation of $\Gamma$ which is a consequence of Gervais' result in \cite{Gervais}.
\begin{lemma}
For any oriented surface $S$ of genus $g\ge 2$ and $s\ge4$ punctures, the mapping class group has the following presentation:
\begin{enumerate}[(1)]
 \item The generators are the Dehn twists $D_a$ along all non separating simple close geodesics $a$ in $S$ ;
 \item The relation between them are the following:
        \begin{enumerate}[(a)]
          \item The type-0 braid relation: for each pair of disjoint non-separating simple closed geodesics $a$ and $b$, we have that $D_aD_b=D_bD_a$ ;
          \item The type-1 braid relation: for each pair of non-separating simple closed geodesics $a$ and $b$ with the geometric intersection number $i(a,b)=1$, we have that: $$D_aD_bD_a=D_bD_aD_b;$$
          \item The lantern relation: for each four-holes sphere embedded in the surface whose boundary $a_0,a_1,a_2,a_3$ are the non-separating simple closed geodesics, we have that: $$D_{a_0}D_{a_1}D_{a_2}D_{a_3}=D_{a_{12}}D_{a_{23}}D_{a_{13}};$$
          \item The chain relation: for each two-holed torus embedded in the surface, we have that: $$(D_aD_bD_c)^4=D_eD_f;$$
          \item The puncture relation: for each sphere with three holes and one puncture embedded in the surface, we have $$D_{a_1}D_{a_2}D_{a_3}=D_{a_{12}}D_{a_{13}}D_{a_{23}}.$$
         \end{enumerate}
\end{enumerate}
\end{lemma}

By using this presentation, we are able to prove the following proposition which is the main step of the proof of Theorem \ref{main}:
\begin{proposition}\label{pro}
 By using the Chekhov-Fock quantization, we obtain a central extension of $\Gamma$ with the following presentation:
\begin{enumerate}[(1)]
 \item Generators:
        \begin{enumerate}[(a)]
         \item One central element: $w=z^{-12}$, where $z$ is the constant coming from the Chekhov-Fock quantization;
         \item One element $\widetilde{D}_a$ associated to each the Dehn twists $D_a$ along all non separating simple close geodesics $a$ in $S$.
        \end{enumerate}
 \item Relations:
       \begin{enumerate}[(a)]
 \item The type-0 braid relation: $\widetilde{D}_a\widetilde{D}_b=\widetilde{D}_b\widetilde{D}_a;$
 \item The type-1 braid relation: $\widetilde{D}_a\widetilde{D}_b\widetilde{D}_a=\widetilde{D}_b\widetilde{D}_a\widetilde{D}_b;$
 \item The Lantern relation: $\widetilde{D}_{a_0}\widetilde{D}_{a_1}\widetilde{D}_{a_2}\widetilde{D}_{a_3}=\widetilde{D}_{a_{12}}\widetilde{D}_{a_{23}}\widetilde{D}_{a_{13}};$
 \item The chain relation: $(\widetilde{D}_a\widetilde{D}_b\widetilde{D}_c)^4=w^{12}\widetilde{D}_e\widetilde{D}_f;$
 \item The puncture relation: $\widetilde{D}_{a_1}\widetilde{D}_{a_2}\widetilde{D}_{a_3}=w\widetilde{D}_{a_{12}}\widetilde{D}_{a_{13}}\widetilde{D}_{a_{23}}.$ 
\end{enumerate}
\end{enumerate}
\end{proposition}

We prove Proposition \ref{pro} by proving a sequence of lemmas. 
\begin{lemma}
 For the type-0 braid relation, we have that $\widetilde{D}_a\widetilde{D}_b=\widetilde{D}_b\widetilde{D}_a$.
\end{lemma}
\begin{proof}
The lifts of two commutative elements are commutative in the central extension.
\end{proof}
\begin{lemma}
 For the type-1 braid relation, by choosing the lifts, we have that: $$\widetilde{D}_a\widetilde{D}_b\widetilde{D}_a=\widetilde{D}_b\widetilde{D}_a\widetilde{D}_b.$$
\end{lemma}
\begin{proof}
Let $a$ and $b$ be two non-separating simple closed geodesics in $\Sigma$ with $i(a,b)=1$. Suppose that we have their lifts $\widetilde{D}_a$ and $\widetilde{D}_b$ such that: $$\widetilde{D}_b\widetilde{D}_a\widetilde{D}_b=z^k\widetilde{D}_a\widetilde{D}_b\widetilde{D}_a.$$

Then by changing the lift $\widetilde{D}_b$ to $\widetilde{D}'_b=z^k\widetilde{D}_b$ we have that:
$$\widetilde{D}_a\widetilde{D}'_b\widetilde{D}_a=\widetilde{D}'_b\widetilde{D}_a\widetilde{D}'_b.$$

Let $x$ and $y$ be another pair of non-separating simple closed geodesics with $i(x,y)=1$. Then there is a homeomorphism $\phi$ of $\Sigma$ sending $a$ and $b$ to $x$ and $y$ respectively. Let $\widetilde\phi$ be its arbitrary lift. Consider the lifts of $D_x$ and $D_y$ as following:
\begin{eqnarray*}
 \widetilde{D}_x=\widetilde\phi^{-1}\widetilde{D}_a\widetilde\phi, \\
 \widetilde{D}_y=\widetilde\phi^{-1}\widetilde{D}'_b\widetilde\phi.
\end{eqnarray*}

Then the corresponding type-1 braid relation is
$$\widetilde{D}_x\widetilde{D}_y\widetilde{D}_x=\widetilde{D}_y\widetilde{D}_x\widetilde{D}_y.$$

We begin by choosing a lift for $D_a$, then there is a unique lift of $D_b$ satisfying the trivial type-1 braid relation. By using homeomorphisms of $\Sigma$, the lifts of the other $D_x$ and $D_y$ appearing in the type-1 braid relations can be fixed too. Then we have the trivial type-1 braid relation everywhere.
\end{proof}

Consider Dehn twists as automorphisms in Ptolemy groupoid. Then they can be expressed as compositions of morphisms corresponding to a sequence of flips and permutations. By using the almost linear representation $K$, we obtain the lifts of Dehn twists. Notice that a Dehn twist may have several expressions different from each other by the relations in Proposition \ref{presentation}. Different expressions may induce different lifts. In the following part, we will prove that by choosing the expression carefully the induced lifted Dehn twists satisfy the relations in Proposition \ref{pro}. 

We will use $\widetilde{F}_\alpha$ to denote $K(F_\alpha)$. For the underlined parts in the following proofs, we use either pentagon relations or commutation relations.
\begin{lemma}
By choosing carefully the lifts for each Dehn twist in the lantern relation, we have that: $$\widetilde{D}_{a_0}\widetilde{D}_{a_1}\widetilde{D}_{a_2}\widetilde{D}_{a_3}=z^{-12}\widetilde{D}_{a_{12}}\widetilde{D}_{a_{23}}\widetilde{D}_{a_{13}}.$$
\end{lemma}
\begin{proof}
 The proof is similar to the one in \cite{Funar}. Consider the the four-holed sphere with one puncture on each boundary component. The ideal triangulation and the labels are given as following:
\begin{center}
\includegraphics[scale=0.3]{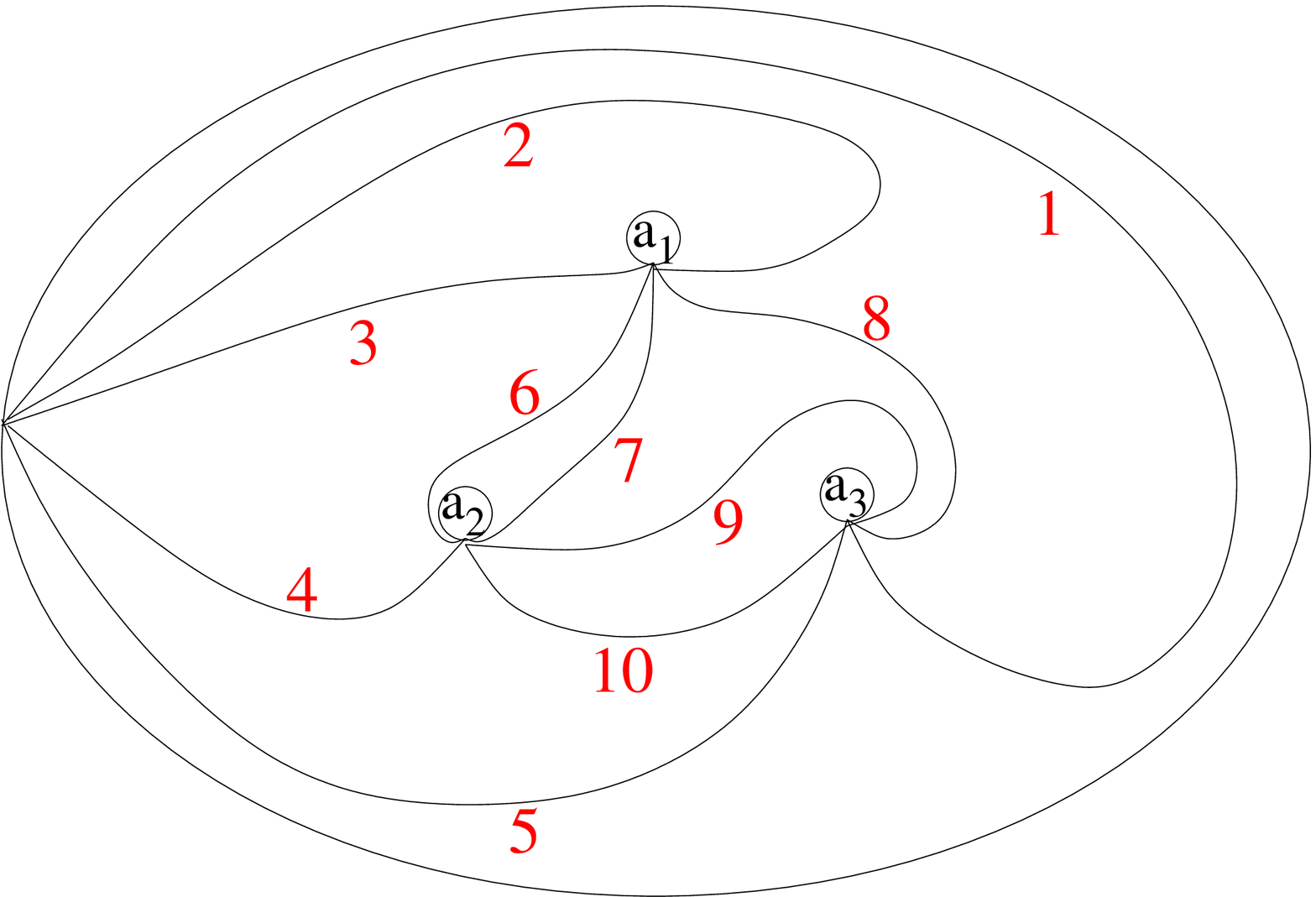}
\end{center}

Then the Dehn twist $D_0, D_1, D_2$ and $D_3$have the following expression:
\begin{eqnarray*}
D_0&=&F_5F_4F_3F_2
\left( \right).
\end{eqnarray*}

For $D_{12}, D_{23}$ and $D_{13}$, we use the same strategy as in \cite{Funar}. First transform the triangulation so that there are only two arcs intersecting the geodesic associated to the Dehn twist. Then do the Dehn twist and transform the triangulation back. Follow this idea we get the following formulas:

\includegraphics[scale=0.3]{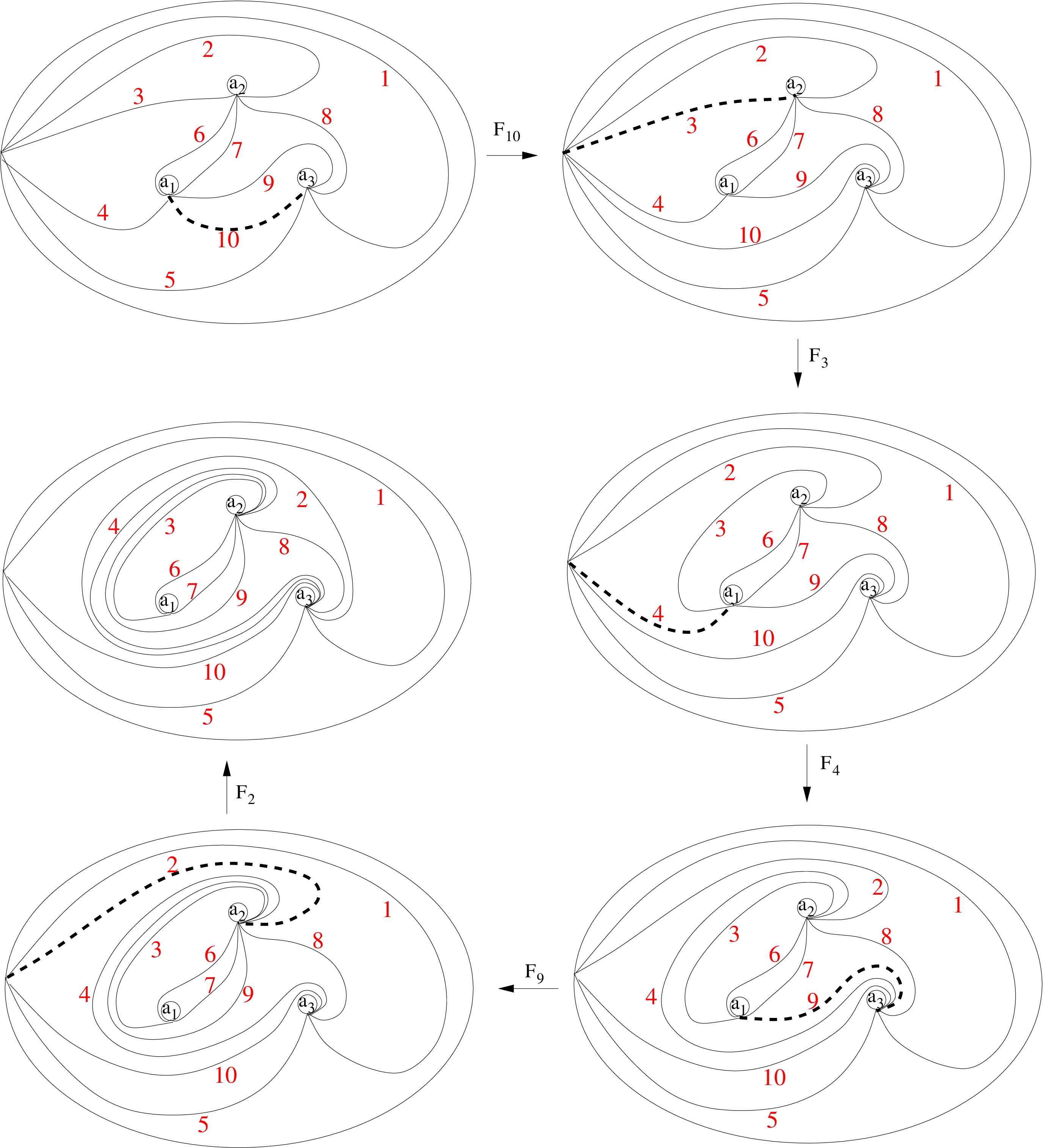}

\begin{equation*}
D_{12}=\rm{Ad}(F_{10}F_3F_4F_9F_2)F_8
\left( %
 \right),\\
\end{equation*}
\begin{center}
\includegraphics[scale=0.3]{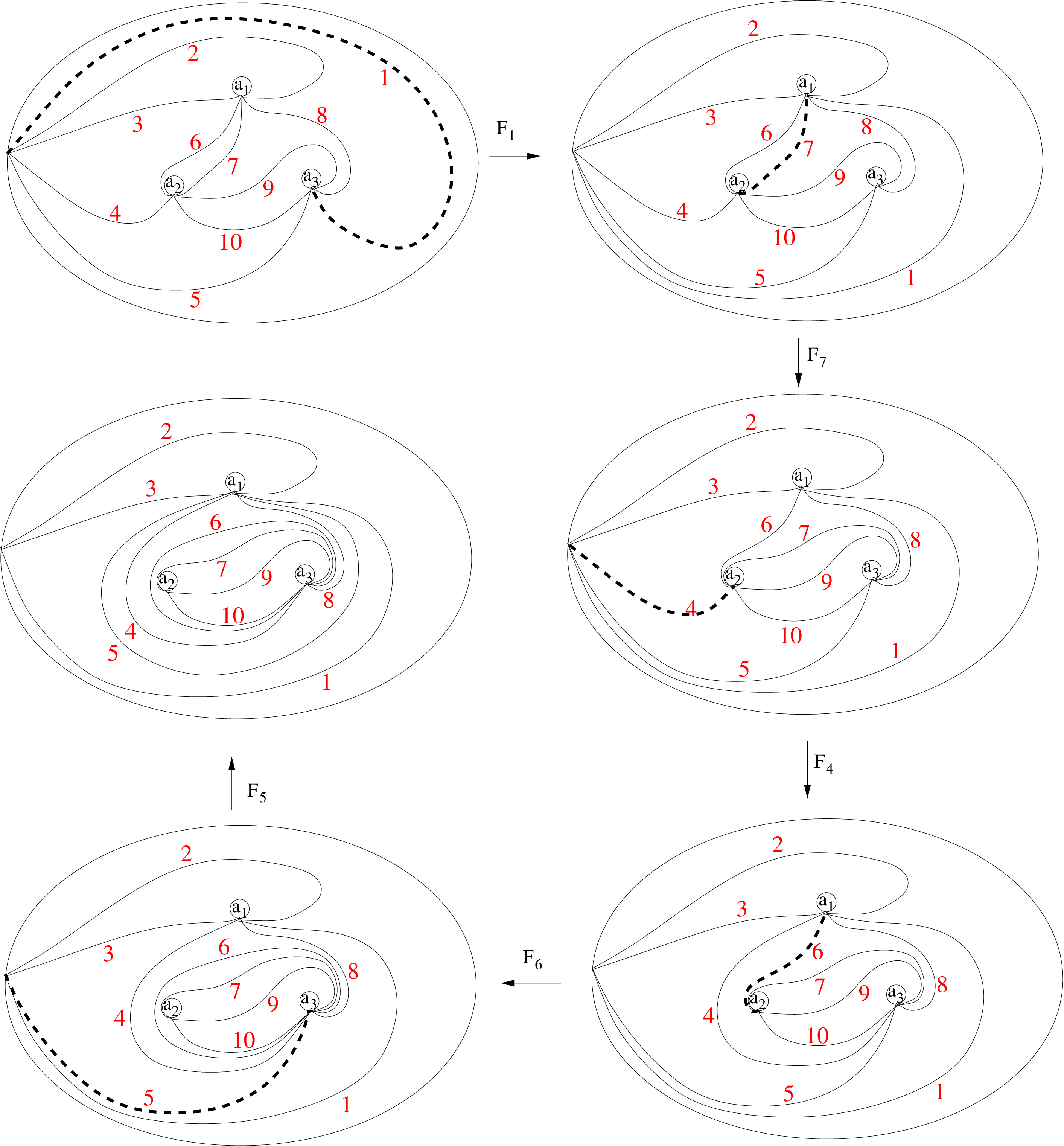}
\end{center}
\begin{equation*}
D_{13}=\rm{Ad}(F_1F_7F_4F_6F_5)F_4
\left( %
 \right),\\
\end{equation*}
\begin{center}
\includegraphics[scale=0.3]{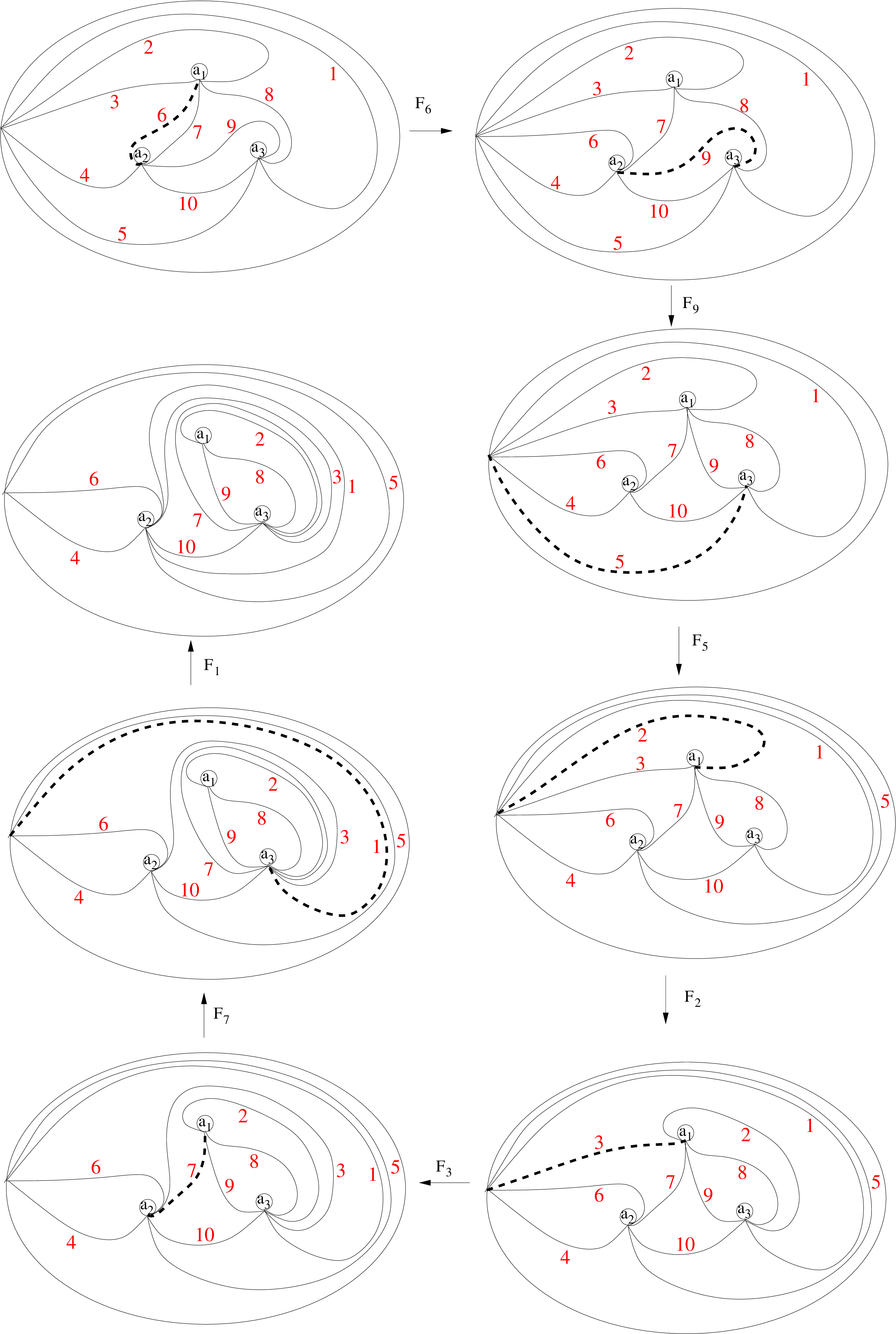}
\end{center}
\begin{equation*}
D_{23}=\rm{Ad}(F_6F_9F_5F_2F_3F_7F_1)F_{10}
\left( %
 \right).\\
\end{equation*}

They can be simplified as following:
\begin{multline*}
D_{12}=F_{10}F_3F_4F_9F_2F_8\left( %
 \right).
\end{multline*}

The expressions which we use to lift Dehn twists are the following: 
\begin{eqnarray*}
&&D_0=F_5F_4F_3F_2
\left( %
 \right).\\
\end{eqnarray*}

Then the composition of the lifts $\widetilde{D}_{12}$, $\widetilde{D}_{13}$ and $\widetilde{D}_{23}$ is the following:
\begin{eqnarray*}
\widetilde{D}_{12}\widetilde{D}_{23}\widetilde{D}_{13}
&=&\widetilde{F}_3\widetilde{F}_8\widetilde{F}_4\widetilde{F}_{10}\widetilde{F}_2\widetilde{F}_4\widetilde{F}_9\widetilde{F}_3\left( %
 \right)=\\
&=&z^{-12}\widetilde{D}_2\widetilde{D}_1\widetilde{D}_0\widetilde{D}_3.
\end{eqnarray*} 
\end{proof}

\begin{lemma}
There are the lifts of Dehn twists such that for the puncture relation we have that: $$\widetilde{D}_{a_1}\widetilde{D}_{a_2}\widetilde{D}_{a_3}=z^{-12}\widetilde{D}_{a_{12}}\widetilde{D}_{a_{13}}\widetilde{D}_{a_{23}}.$$ 
\end{lemma}
\begin{proof}
As this is a degenerated case of the lantern relation, the proof is the same as above.
\end{proof}

\begin{lemma}
By choosing carefully the lifts of the Dehn twists in the chain relation, we have that: $$(\widetilde{D}_a\widetilde{D}_b\widetilde{D}_c)^4=z^{-24}\widetilde{D}_e\widetilde{D}_f.$$
\end{lemma}
\begin{proof}
Consider the two-holed torus with one puncture on each boundary component. The simple closed geodesics $a$, $b$, $c$, $e$ and $f$ are as following:
\begin{center}
\includegraphics[scale=0.5]{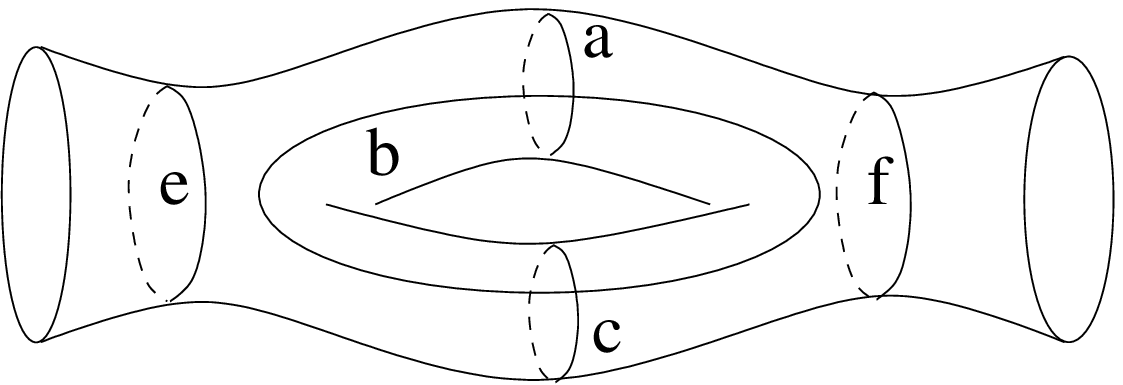}
\end{center}

We consider the following ideal triangulation of the two-holed torus:
\begin{center}
\includegraphics[scale=0.4]{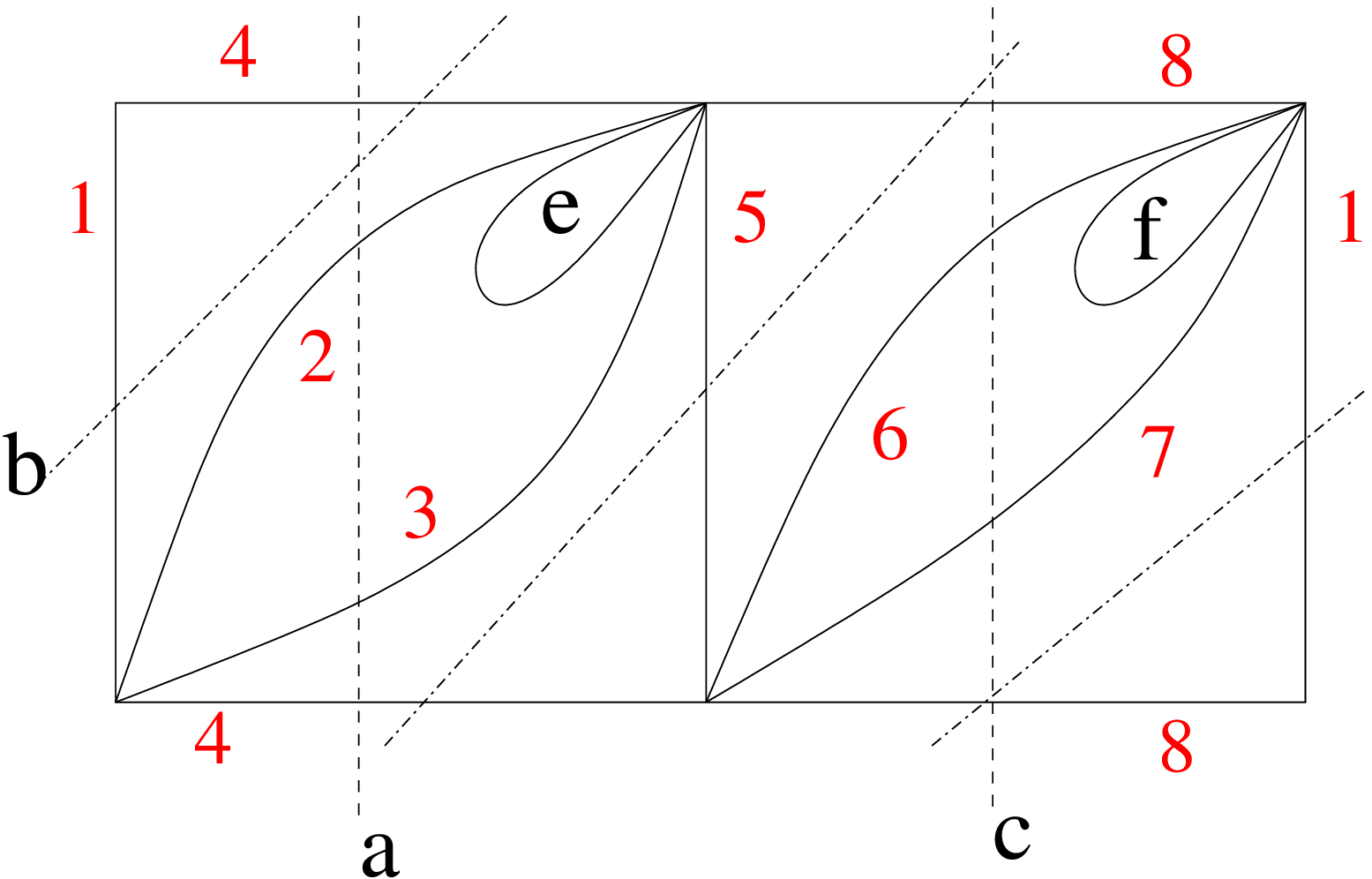}
\end{center}

The Dehn twists corresponding to $a$, $b$, $c$, $e$ and $f$ can be presented as following:
\begin{eqnarray*}
D_{a}&=&F_3F_4F_3
\left( \right).
\end{eqnarray*}

The Dehn twist $D_e$ can be simplified as following:
\begin{eqnarray*}
D_{e}&=&F_5F_4F_8F_6F_7F_3
\left( %
 \right),\\
\end{eqnarray*}
and $D_f$ can be simplified in a similar way as following:
\begin{eqnarray*}
D_{f}&=&F_7F_1F_4F_3F_2F_7F_8\left( %
 \right),
\end{eqnarray*}
by sending $(2,3,4,5,6,7,8)$ to $(6,7,8,1,2,3,4)$.

We will use the following expression of these Dehn twists to get their lifts: 
\begin{eqnarray*}
D_{a}&=&F_3F_4F_3
\left( %
 \right).
\end{eqnarray*}

We need to check if the induced lifts of the Dehn twists satisfy the trivial type-1 braid relation. As there is a symmetry between $D_a$ and $D_c$ with respect to $D_b$, we only need to verify the trivial type-1 braid relation for $\widetilde D_a$ and $\widetilde D_b$.

We rewrite $\widetilde D_b$ as following :
\begin{multline*}
 \widetilde{D}_{b}=\widetilde{F}_8\widetilde{F}_1\widetilde{F}_5\widetilde{F}_8
\left( %
 \right).
\end{multline*}

By using once pentagon relation, we have the following equality:
\begin{eqnarray*}
\widetilde{D}_a=\widetilde{F}_3\widetilde{F}_4\widetilde{F}_3\left( %
 \right).
\end{eqnarray*}

Then the type-1 braid relation can be verified as following:
\begin{eqnarray*}
\widetilde{D}_b\widetilde{D}_a\widetilde{D}_b&=&z^{-7}\widetilde{F}_1\widetilde{F}_5\widetilde{F}_8\widetilde{F}_4\left( %
 \right)=\\
&=&\widetilde{D}_a\widetilde{D}_b\widetilde{D}_a.
\end{eqnarray*}

The following computations will be used in the later proof:
\begin{eqnarray*}
\widetilde{D}_c&=&\widetilde{F}_7\widetilde{F}_8\widetilde{F}_7\left( %
 \right).
\end{eqnarray*}

Taking the product of $\widetilde{D}_c$, $\widetilde{D}_b$ and $\widetilde{D}_a$, we have the following:
\begin{eqnarray*}
\widetilde{D}_c\widetilde{D}_b\widetilde{D}_a&=&z^{-5}\widetilde{F}_8\widetilde{F}_7\left( %
 \right).
\end{eqnarray*}

By taking the square, we have :
\begin{eqnarray*}
(\widetilde{D}_c\widetilde{D}_b\widetilde{D}_a)^2
&=&z^{-10}\widetilde{F}_8\widetilde{F}_7\widetilde{F}_1\widetilde{F}_5\underline{\widetilde{F}_4}\widetilde{F}_3\left( %
 \right).
\end{eqnarray*}

This implies that :
\begin{eqnarray*}
(\widetilde{D}_c\widetilde{D}_b\widetilde{D}_a)^4&=&z^{-22}\widetilde{F}_7\widetilde{F}_8\widetilde{F}_1\widetilde{F}_3\widetilde{F}_5\underline{\widetilde{F}_2}\widetilde{F}_6\left( %
 \right)=\\
&=&z^{-24}\widetilde{D}_f\widetilde{D}_e.
\end{eqnarray*}

The chain relation becomes:
\begin{equation*}
 (\widetilde{D}_a\widetilde{D}_b\widetilde{D}_c)^4=z^{-24}\widetilde{D}_e\widetilde{D}_f
\end{equation*}
\end{proof}

We say that the lifts of Dehn twists are \textbf{normalized} if all braid type relations and all lantern relations are lifted in the trivial way.

By the lemmas above, we normalize the lift of each Dehn twist $D$ from $\widetilde D$ to $z^{-12}\widetilde D$. Then we get the presentation in the proposition (\ref{pro}).

The cohomological arguments from \cite{Funar} settle our theorem.
\bibliographystyle{plain}
\bibliography{QuantumFock}
\end{document}